\documentclass[twoside, 12pt, reqno]{amsart}

\usepackage[left=2.5cm, right=2.5cm, top=3cm, bottom=2.5cm]{geometry}

\usepackage[colorlinks=true, pdfstartview=FitV, linkcolor=blue,citecolor=blue, urlcolor=blue]{hyperref}

\usepackage[usenames]{xcolor}
\definecolor{labelkey}{rgb}{0,0,1}
\definecolor{Red}{rgb}{0.7,0,0.1}
\definecolor{Green}{rgb}{0,0.7,0}

\usepackage{amsfonts, amssymb, amsmath, amsthm, mathrsfs, bbm, cjhebrew, gensymb, textcomp, mathtools, dsfont,tikz}

\usepackage[normalem]{ulem}

\usepackage{commath, setspace, subcaption, parcolumns, multirow, multicol, accents, comment, marginnote, verbatim, empheq, enumerate, stackrel, enumitem}

\usepackage{cleveref}
\usepackage{graphicx}
\usepackage{epsfig}
\usepackage{psfrag}
\usepackage{float}
\usepackage[active]{srcltx}
\usepackage[pagewise, mathlines]{lineno}

\newtheorem{remark}{Remark}
\newtheorem{lemma}{Lemma}
\numberwithin{lemma}{section}

\newtheorem{theorem}{Theorem}
\numberwithin{theorem}{section}
\newtheorem{definition}{Definition}
\newtheorem{proposition}{Proposition}
\numberwithin{proposition}{section}

\numberwithin{equation}{section}

\newcommand{\al}{\alpha}
\newcommand{\be}{\beta}

\newcommand{\eps}{\epsilon}

\newcommand{\Lam}{\Lambda}
\newcommand{\si}{\sigma}

\newcommand{\tht}{\theta}

\newcommand{\Om}{\Omega}

\newcommand{\lpj}{\triangle_j}
\newcommand{\lpk}{\triangle_k}

\newcommand{\ZZ}{\mathbb{Z}}
\newcommand{\RR}{\mathbb{R}}

\newcommand{\lb}{\big\langle}
\newcommand{\rb}{\big\rangle}
\newcommand{\lbn}{\langle}
\newcommand{\rbn}{\rangle}

\newcommand{\Sob}[2]{\lVert#1\rVert_{#2}}
\newcommand{\nrm}[1]{\lVert#1\rVert}

\newcommand{\bdy}{\partial}

\newcommand{\til}[1]{\widetilde{#1}}

\newcommand{\Hdot}{\dot{H}}

\newcommand{\Acal}{\mathcal{A}}
\newcommand{\Bcal}{\mathcal{B}}

\DeclareMathOperator{\supp}{supp}

\makeatother

\newcommand{\la}{\Lambda}

\newcommand{\TT}{\mathbb{T}}

\newcommand{\gam}{\gamma}

\title[Existence of weak solutions to the generalized SQG equations in Sobolev spaces]{Existence of weak solutions to the generalized SQG equations in Sobolev spaces}
\author{Anuj Kumar$^{1,\dagger}$}
\address{$^1$Department of Mathematics,
Indian Institute of Technology Jodhpur, India}
\address{$\dagger$ corresponding author}
\email[A Kumar]{akumar241@outlook.com}
\date{\today}
\thanks{}
\begin{document}
\begin{abstract}
	We consider the two-dimensional dissipative generalized surface quasi-geostrophic equations given by
    \[\bdy_t \tht+u\cdot \nabla \tht+\nu (-\Delta)^{\al}\tht=0,\quad u=\nabla^{\perp}\Lam^{\be-2}\tht\] and establish global existence of weak solutions with initial data in $\dot{H}^{\frac{\be}{2}-1}(\mathbb{R}^2)$. Our result extends the framework of Marchand's solutions (Comm. Math. Phys. 277 (2008), no. 1, 45–67) to this larger family of active scalar equations.
\end{abstract}
\maketitle
{\noindent \small {\it {\bf Keywords: Generalized surface quasi-geostrophic (gSQG) equation, Weak solution, Sobolev spaces}
}  }

\section{Introduction}
In this paper, we study the following generalized surface quasi-geostrophic (gSQG) equations:
\begin{align}\label{gSQG}
    \begin{cases}
        &\bdy_t \tht+u\cdot \nabla \tht+\nu (-\Delta)^{\al}\tht=0,\\
       &u=\mathcal{R}^{\perp}\Lam^{\be-1}\tht=\left(\bdy_{x_2}\Lam^{\beta-2}\tht,-\bdy_{x_1}\Lam^{\beta-2}\tht\right),\\
        &\tht(x,0)=\tht_0(x),\quad x\in \mathbb{R}^2,\,t>0.
    \end{cases}
\end{align}
In the above equation, $\tht(x,t)$ denotes the evolving scalar, $u(x,t)$ denotes the advecting velocity field, $\al \in (0,1)$ denotes the dissipation parameter, and $\beta\in (0,2)$ denotes the constitutive law parameter.  $\Lam$ denotes the fractional Laplacian operator $(-\Delta)^{1/2}$, defined via Fourier multiplier
$\widehat{\Lam f}(\xi)=|\xi|\widehat{f}(\xi)$
where $\widehat{f}=\mathcal{F}(f)$ is the Fourier transform of $f$ as defined in \eqref{def:FT}.\par
The study of the Cauchy problem for the family of equations in \eqref{gSQG} was initiated in \cite{ChaeConstantinWu2011} and \cite{ChaeConstantinCordobaGancedoWu2012}, while its inviscid counterpart was first studied in \cite{ChaeConstantinWu2012}. For $\be\in[0,1]$, \eqref{gSQG} interpolates between the 2D Euler equation in vorticity form $(\be=0)$ and the dissipative quasi-geostrophic equation (SQG) $(\be=1)$. For $\be \in (1,2)$, it represents a family of active scalar equations with increasingly singular constitutive laws. The SQG equation is a fundamental equation in geophysics, derived from general quasi-geostrophic models of atmospheric and ocean fluid flows. It represents the evolution of temperature or buoyancy of a strongly stratified fluid in a rapidly rotating regime. The mathematical study of the inviscid case was initiated in \cite{ConstantinMajdaTabak1994}, where it was shown to be an analog to the 3D Euler equations, in particular exhibiting a similar vortex-stretching mechanism. The SQG also exhibits turbulent features analogous to the 2D Euler equation as observed in \cite{MajdaTabak1996}. Depending on the balance between the dissipative and the nonlinear term, one distinguishes between three regimes of the SQG equation: subcritical ($\al>1/2$), critical ($\al=1/2$), and supercritical ($\al<1/2$). Smooth solutions are known to exist globally in the subcritical \cite{ConstantinWu1999, Resnick1995} and critical regimes, wherein the problem of global regularity remained open until it was settled independently by several sophisticated methods \cite{KiselevNazarovVolberg2007, CaffarelliVasseur2010, ConstantinVicol2012}. 
In the supercritical regime, the global regularity of smooth solutions remains an outstanding open problem. Nevertheless, several results on local well-posedness for large data, global well-posedness for small data, and Gevrey-class smoothing of solutions in time are available \cite{Miura2006, ChenMiaoZhang2007, HmidiKeraani2007, Wu2004, Biswas2014, Dong2010}. Similar studies for the gSQG equation \eqref{gSQG} have also been conducted; see, for instance, \cite{ChaeConstantinWu2011, ChaeConstantinWu2012, ChaeConstantinCordobaGancedoWu2012, HuKukavicaZiane2015, JollyKumarMartinez2020a} and the references therein.
\par
In this paper, we study the existence of global weak solutions to \eqref{gSQG} with initial data in $\Hdot^{\frac{\be}{2}-1}(\RR^2)$. The existence of weak solutions for the SQG equation was first studied by Resnick in \cite{Resnick1995}, where global $L^2$-weak solutions were established for both the inviscid and dissipative cases. Later on, Marchand \cite{Marchand2008} extended Resnick's result to a more general class and proved the global existence of weak solutions in $L^p(\RR^2)$ for $p>4/3$ and in $\Hdot^{-1/2}(\RR^2)$. The uniqueness of weak $L^2$-solutions is an unsolved problem. In this direction, Buckmaster, Shkoller and Vicol \cite{BuckmasterShkollerVicol2019} have shown nonuniqueness for a class of solutions with negative Sobolev regularity. For gSQG equations \eqref{gSQG}, the global existence of $L^2$-weak solutions have been established for $\be\in (0,1]$ and also for $\be \in (1,2)$ (see \cite{ChaeConstantinCordobaGancedoWu2012, MiaoXue2011JDE,LazarXue2019,FerreiraNichePlanas2017}. In this paper, we prove the existence of global weak solutions in the larger class of functions $\Hdot^{\frac{\be}{2}-1}(\RR^2)$ for the entire range of gSQG equations, i.e., $\be\in(0,2)$. In particular, we recover Marchand's result for the SQG equation ($\be=1$). In a recent work by Mengual and Solera \cite{MengualSolera2026}, nonuniqueness of weak solutions in $\Hdot^{\frac{\be}{2}-1}(\RR^2)$ has also been shown for certain classes of gSQG equations. It would also be interesting to extend Marchand's framework of $L^p$-weak solutions to \eqref{gSQG}, which we suppose requires commutator estimates in $L^p$-spaces for non-zero order singular operators that appear in the constitutive laws of \eqref{gSQG}.  \par
To establish our main result, we make a crucial observation about the structure of the nonlinear term in the spirit of Marchand's remark in \cite{Marchand2008}. This observation allows us to reformulate the nonlinear term in terms of commutators. As a result, the nonlinear term can be interpreted as a well-defined tempered distribution. The Fourier-based formulations of Sobolev spaces and Plancherel's theorem also play a crucial role in our proof. Due to the apparent lack of standalone commutator estimates for non-zero order singular operators, we prove bounds for commutators as they appear in the arguments, i.e., in terms of frequency-localized trilinear terms. Using their Fourier-based representations, we prove estimates for the commutators in terms of Sobolev norms by performing frequency decompositions with Littlewood-Paley dyadic operators. 
\par
Our main result establishing global weak solutions to \eqref{gSQG} is stated in \cref{T:main}. We provide a brief outline of the proof: First, we regularize \eqref{gSQG} by adding an artificial viscosity term, which ensures global smooth solutions for the approximation. The main difficulty in showing that the limiting solution satisfies \eqref{gSQG} in the weak sense lies in proving convergence of the nonlinear terms of the approximating equation in the sense of distributions. By using the crucial observation on the nonlinear structure, we express certain nonlinear terms in terms of commutators, thus affording better control over them and giving the required convergence upon passage to the limiting solution. 
\begin{theorem}\label{T:main}
    Let $T\in (0,\infty]$, $\al \in (0,1]$, and $\be \in (0,2)$. If $\tht_0 \in \Hdot^{\frac{\be}{2}-1}(\RR^2)$, then there exists a weak solution $\tht(x,t)$ on $(0,T)\times \RR^2$ to \eqref{gSQG} such that
    \[\tht \in L^\infty ((0,T);\Hdot^{\frac{\be}{2}-1}(\RR^2))\cap L^2 ((0,T);\Hdot^{\al+\frac{\be}{2}-1}(\RR^2)).\]
    Furthermore, $\tht$ satisfies the energy inequality 
    \begin{align*}
        \Sob{\tht(t)}{\Hdot^{\frac{\be}{2}-1}(\RR^2)}^2+2\nu \int_0^t \int_{\RR^2}|\Lam^{\al+\frac{\be}{2}-1}\tht(s)|^2\,dx\,ds\le \Sob{\tht_0}{\Hdot^{\frac{\be}{2}-1}(\RR^2)}^2
    \end{align*}
    for all $t\in (0,T)$.
\end{theorem}


\section{Mathematical Preliminaries}
 In this section, we define the relevant notations and definitions used in the paper.
Let $\mathscr{S}(\RR^2)$ denote the space of Schwartz class functions on $\RR^2$ and $\mathscr{S}'(\RR^2)$ denote the space of tempered distributions. We denote by $\hat{f}$ or $\mathcal{F}(f)$, the Fourier transform of $f$, defined by
	\begin{align}\label{def:FT}
	    \hat{f}(\xi)\overset{}{:=}\int_{\RR^2} e^{-2\pi i x\cdot \xi}f(x)\,dx,\quad f\in\mathscr{S}'(\RR^2).
	\end{align}
   We have for any $f,g \in L^2(\RR^2)$
	    \begin{align}\notag
	        \lbn f,g \rbn_{L^2}=\lbn\hat{f},\hat{g}\rbn_{L^2},
	    \end{align}
        where $\lb ,\rb$ denotes the $L^2-$ inner product.
   $\Lam^\si$ denotes the fractional Laplacian operator, which is defined by its Fourier multiplier as
	    \begin{align}\notag
	        \mathcal{F}(\Lam^\si f)(\xi)=|\xi|^\si\mathcal{F}(f),\quad \si\in\RR.
	    \end{align} 
     Next, we define the $L^2$-based homogeneous and inhomogeneous Sobolev spaces. For $\si \in \mathbb{R}$, we have
	    \begin{align}
	        &\Hdot^\si(\mathbb{R}^2){:=}\left\{f\in \mathscr{S}(\RR^2):\hat{f}\in L^2_{loc},\, \Sob{f}{\Hdot^\si}:=\Sob{\Lam^\si f}{L^2}<\infty\right\},\label{def:hom:Sob:norm}\\
	        &H^\si(\mathbb{R}^2){:=}\left\{f\in \mathscr{S}(\RR^2):\hat{f}\in L^2_{loc}, \Sob{f}{H^\si}{=}\Sob{(I-\Delta)^{\si/2}) f}{L^2}<\infty\right\}.\label{def:inhom:Sob:norm}
	    \end{align}
Next, we provide a brief overview of the Littlewood-Paley decomposition of functions. 
We define
	\begin{align*}
	\mathscr{Q}(\RR^2){:=}\left\{f\in \mathscr{S}(\RR^2): \int_{\RR^2} f(x)x^{\tau}\, dx=0, \quad \abs{\tau}=0,1,2,\cdots \right\}.
	\end{align*}
	We denote by $\mathscr{Q}(\RR^2)'$, the topological dual of $\mathscr{Q}(\RR^2)$. Then, $\mathscr{Q}(\RR^2)'$ can be identified with the space of tempered distributions modulo the vector space of polynomials on $\mathbb{R}^2$, which we denote by $\mathscr{P}$, i.e. 
	\begin{align*}
	\mathscr{Q}'(\RR^2)\cong\mathscr{S}(\RR^2)/\mathscr{P}.
	\end{align*}
	Let us denote by ${\Bcal}(r)$ the open ball centered at the origin with radius $r$ and ${\Acal}(r_{1},r_{2})$, the open annulus centered at the origin with inner and outer radii $r_{1}$ and $r_{2}$. We can find two non-negative radial functions denoted $\chi,\varphi\in\mathscr{S}(\RR^2)$ satisfying
    \[\supp\chi\subset{\Bcal}(1)\] and \[\supp\varphi\subset{\Acal}(2^{-1},2)\] and such that the following conditions are satisfied
	\begin{align}
	    \begin{cases}
	    \sum_{j\in\ZZ}\varphi(2^{-j}\xi)=1,\\
	    \chi(\xi)+\sum_{j\geq0}\varphi(2^{-j}\xi)\equiv 1,\,\forall \xi\in\RR^2\setminus\{\mathbf{0}\},
	    \end{cases}\label{LP:conditions}
	\end{align}
    We define
    \[\varphi_j(\xi){:=}\varphi(2^{-j}\xi),\quad \chi_j(\xi){:=}\chi(2^{-j}\xi)\]
    and
    \begin{align}\notag            {\Acal}_{j}:= {\Acal}(2^{j-1},2^{j+1}),\quad{\Acal}_{\ell,k}:= {\Acal}(2^{\ell},2^{k}),\quad  {\Bcal}_j:={\Bcal}(2^j).
        \end{align}
    We observe that
	    \begin{align}\label{rewrite:supp}
	        \supp\varphi_j\subset{\Acal}_j,\quad \supp\chi_j\subset{\Bcal}_j.
	    \end{align}
    With the above assumptions, the following almost-orthogonality conditions hold
    \[\supp\varphi_i\cap\supp\varphi_j=\varnothing,\,\text{if}\,
	    |i-j|\geq2,\quad
	    \text{and}\quad\supp\varphi_i\cap\supp\chi =\varnothing, \,\text{if}\,i\ge 1.\]
	Let us denote by ${\lpj}$ and $S_{j}$, the (homogeneous) Littlewood-Paley dyadic blocks, defined by
	\begin{align}\notag
	\mathcal{F}({\lpj}f)=\varphi_{j}\mathcal{F}(f), \quad \mathcal{F}(S_{j}f)=\chi_{j}\mathcal{F}(f). 
	\end{align}
    We denote by 
    \[H_j f=(I-S_j)f.\]
    Note that, by \eqref{LP:conditions}, we have
	\begin{align*}
	&\mathcal{F}({\lpj}f)|_{{\Acal}_j^c}=0,\quad
	\mathcal{F}(S_{j}f)|_{{\Bcal}_j^{c}}=0,
	\end{align*}
    Furthermore, for any $f \in \mathscr{S}(\RR^2)$, we have
    \begin{align*}
         f&=S_if+\sum_{j\geq i}\lpj f,\quad i\in\ZZ.
    \end{align*}
	and for any $f\in\mathscr{Q}(\RR^2)'$, we have
	    \begin{align*}
	        f&=\sum_{j\in\ZZ}\lpj f.
	    \end{align*}
It is straightforward to verify the following norm equivalence
	    \begin{align}\label{norm:equivalence}
	       C^{-1}\Sob{f}{L^2(\RR^2)}\leq \left(\sum_{j\in \mathbb{Z}}\nrm{\lpj f}_{L^2(\RR^2)}^{2}\right)^{\frac{1}{2}}\leq C\Sob{f}{L^2(\RR^2)},
	    \end{align}
	for some constant $C$. 
    We recall the classical Bernstein inequalities (see \cite{BahouriCheminDanchinBook2011}):
    \begin{lemma}\label{L:Bernstein}
		Let $\si\in\RR$ and $1\le p \le q\le \infty$. Then
		\begin{align*}
	C^{-1}2^{\si j}\nrm{{\lpj}f}_{L^q(\mathbb{R}^2)}\le \nrm{\Lam^{\si}{\lpj}f}_{L^q(\mathbb{R}^2)}\le C 2^{\si j+2j(\frac{1}{p}-\frac{1}{q})}\nrm{{\lpj}f}_{L^p(\mathbb{R}^2)},
		\end{align*}
	 where $C>0$ is a constant that depends on $p,q$ and $\si$.
	\end{lemma}
    We also recall the following classical product estimate in homogeneous Sobolev spaces \cite{BahouriCheminDanchinBook2011}:
\begin{lemma}\label{L:Sobolev:product}
    Suppose that $s,t<1$ and $s+t>0$. Let $f\in \Hdot^{s}(\RR^2)$ and $g\in \Hdot^{t}(\RR^2)$. Then
    \[\Sob{fg}{\Hdot^{s+t-1}(\RR^2)}\le C\Sob{f}{\Hdot^s(\RR^2)}\Sob{g}{\Hdot^t(\RR^2)}.\]
\end{lemma} 
As mentioned earlier, to control the trilinear terms appearing in the proof of \cref{T:main}, we will reformulate them in terms of commutators which are then estimated using general estimates as stated in \cref{L:FJKM26lemma}, \cref{L:JKM22lemma}, and \cref{L:mainlemma}.
We remark that \cref{L:JKM22lemma} for the particular case of $i=j$ was proved in \cite{JollyKumarMartinez2020b}. The proof in the current setting of $i\sim j$ is almost identical and is therefore omitted to avoid redundancy. The proof of \cref{L:FJKM26lemma} in the setting of $\Om=\TT^2$ was presented in \cite{FuJollyKumarMartinez2026}. The proof in the present case of $\Om=\RR^2$ follows by a similar argument. 

{
\begin{lemma}
\label{L:FJKM26lemma}
    For $s\in (0,1)$, let $\rho \in [0,s]$. Let $f,g,h\in \mathcal{S}(\mathbb{R}^2)$. Given $i, j\in \mathbb{Z}$ such that $|i-j|\le k$ for a positive integer $k$, suppose that $\text{supp} \ \hat{f}\subset \mathcal{A}_i$ and $\text{supp} \ \hat{h}\subset \mathcal{A}_j$. Then there exists a constant $C>0$, depending on $s,k$ and $\rho$ such that
    \[|\lbn [\la^{-s}\nabla,g]f,h\rbn|\le C\Sob{\mathcal{F}(\la^{1-s+\rho} g)}{L^1(\RR^2)}\Sob{\la^{-\rho}f}{L^2(\RR^2)}\Sob{h}{L^2(\RR^2)}.\]
\end{lemma}
}
    \begin{lemma}\label{L:JKM22lemma}
Let $s \in [0,1)$, $\rho \in \mathbb{R}$. Let $f,g,h\in \mathcal{S}(\mathbb{R}^2)$. Given $i, j\in \mathbb{Z}$ such that $|i-j|\le k$ for a positive integer $k$, suppose that $\supp\hat{h}\subset{\Acal}_j$. Then there exists a sequence $\{c_i\}\in\ell^2(\ZZ)$ such that $\Sob{\{c_i\}}{\ell^2(\mathbb{Z})}\leq1$ and
    \begin{align*}
        |\lb [\Lam^{s+\rho+1}\Delta_i,g]f,h\rb|\leq Cc_{i}
              (\Sob{\mathcal{F}(\Lam g)}{L^1(\RR^2)}+\Sob{\Lam^2 g}{L^2(\RR^2)})\Sob{\Lam^s f}{L^2(\RR^2)}\Sob{\Lam^\rho h}{L^2(\RR^2)},
    \end{align*}   
    \begin{align*}
        |\lb[\Lam^{s+\rho}\nabla\Delta_i,g]f,h\rb|\leq Cc_{i}
              (\Sob{\mathcal{F}(\Lam g)}{L^1(\RR^2)}+\Sob{\Lam^2 g}{L^2(\RR^2)})\Sob{\Lam^s f}{L^2(\RR^2)}\Sob{\Lam^\rho h}{L^2(\RR^2)}
    \end{align*}  
hold for some constant $C>0$, depending only on $s,\rho,k$. 
\end{lemma}
Next, we prove an estimate for the commutators as they appear in the nonlinear term, but nevertheless we believe it to be of independent interest. The operators $\Lam^{-\si}\nabla^{\perp}$ correspond to a non-zero-order singular operator, which requires a finer decomposition using Littlewood-Paley operators, along with estimation using \cref{L:FJKM26lemma} and \cref{L:JKM22lemma}.
\begin{lemma}\label{L:mainlemma}
    Let $\si\in (0,1)$ and $\nu \in [0,1-\si)$. Let $f,g \in \mathcal{S}(\RR^2)$. Then
    \begin{align*}
        \Sob{[\Lam^{-\si}\nabla^{\perp}\cdot, \nabla g]f}{\Hdot^{\si+\nu}(\RR^2)}\le C(\Sob{\mathcal{F}(\nabla \Lam g)}{L^1(\RR^2)}+\Sob{\nabla \Lam^2 g}{L^2(\RR^2)})\Sob{f}{\Hdot^{\nu}(\RR^2)}.
    \end{align*}
\end{lemma}
\begin{proof}
    Let $h\in \mathcal{S}(\RR^2)$. Then, 
    \begin{align*}
        \lb [\Lam^{-\si}\nabla^{\perp}\cdot, \nabla g]f, h\rb=\sum_{|j-k|\le 1}\lb\Delta_j \left([\Lam^{-\si}\nabla^{\perp}\cdot, \nabla g]f\right),\Delta_k h\rb.
    \end{align*}
We have the decomposition
\begin{align*}
      \lb [\Lam^{-\si}\nabla^{\perp}\cdot, \nabla g]f, h\rb=&\sum_{|j-k|\le 1}\left \{ \lb [\la^{-\si}\nabla^{\perp}\cdot,\nabla g ]\lpj f,\lpk h \rb \right.\\& \hspace{4 em}  \left.-\lb [\lpj,\nabla g]\cdot \nabla^{\perp} \la^{-\si}f,\lpk h \rb \right.\\& \hspace{4 em}  \left.  +\lb [\lpj,\nabla g]\cdot \nabla^{\perp} f,\lpk \la^{-\si}h\rbn\right\}\\
      =&\sum_{|j-k|\le 1} I_1+I_2+I_3
\end{align*}
which can be verified by expanding each of $I_1$, $I_2$ and $I_3$. We now estimate each of the terms $I_1, I_2,$ and $I_3$ individually.

Applying  \cref{L:FJKM26lemma} with $s=\rho=\si$ and then Bernstein's inequality to find
\begin{align*}
    |I_1|\le& C\|\widehat{\nabla \la g}\|_{L^1(\RR^2)}
    \Sob{\lpj\la^{-\si}f}{L^2(\RR^2)}\Sob{\lpk h}{L^2(\RR^2)} \\
    \le&C\|\widehat{\nabla \la g}\|_{L^1(\RR^2)}
    \Sob{\lpj\la^{-\nu}f}{L^2(\RR^2)}\Sob{\lpk \Lam^{-\si-\nu} h}{L^2(\RR^2)}.
\end{align*}
Next observe that 
\[I_2=\lb [\lpj,\nabla g]\cdot \nabla^{\perp} \la^{-\si}f,\lpk h\rb =-\lbn [\lpj,\nabla g] \la^{-\si}f \cdot,\nabla^{\perp}\lpk h\rb. \]
Applying Lemma \ref{L:JKM22lemma} with $s=\si+\nu,\, \rho=-\si-\nu-1$, we get
\begin{align*}
    |I_2|&\le Cc_j \left(\Sob{ \widehat{\nabla \Lam g}}{L^1(\RR^2)}+\Sob{\nabla \la^{2}g}{L^2(\RR^2)}\right)\Sob{\la^{\nu}f}{L^2(\RR^2)}\Sob{\lpk \la^{\si+\nu}\mathcal{R}^{\perp}h}{L^2(\RR^2)}\\
    &\le Cc_j \left(\Sob{ \widehat{\nabla \Lam g}}{L^1(\RR^2)}+\Sob{\nabla \la^{2}g}{L^2(\RR^2)}\right)\Sob{\la^{\nu}f}{L^2(\RR^2)}\Sob{\lpk \la^{\si+\nu}h}{L^2(\RR^2)},
\end{align*}
{for some $c_j \in \ell^{2}(\mathbb{Z})$}. We next apply Lemma \ref{L:JKM22lemma} with $s=\nu,\, \rho=-1-\nu$ on
\[I_3=-\lb [\lpj,\nabla g]\cdot \nabla^{\perp} f,\lpk \la^{-\si}h \rb=\lb [\lpj,\nabla g] f \cdot,\nabla^{\perp}\lpk \la^{-\si}h\rb,\]
to obtain an identical upper bound on $|I_3|$. We bound the sum over $|j-k|\le 1$ of the resulting estimates from $I_1$, $I_2$, and $I_3$ by applying the Cauchy-Schwarz inequality and using the fact that $c_j \in \ell^{2}(\mathbb{Z})$ to obtain
\begin{align*}
    |\lb [\Lam^{-\si}\nabla^{\perp}\cdot, \nabla g]f, h\rb|\le C(\Sob{\mathcal{F}(\nabla \Lam g)}{L^1(\RR^2)}+\Sob{\nabla \Lam^2 g}{L^2(\RR^2)})\Sob{f}{\Hdot^{\nu}}\Sob{h}{\Hdot^{-\si-\nu}(\RR^2)}.
\end{align*}
Since $h$ was an arbitrary element of $\mathcal{S}(\RR^2)$, this completes the proof.\par
We now state the definition of weak solutions to the dissipative gSQG equations.
\begin{definition}
    A weak solution to \eqref{gSQG} on $(0,T)\times \RR^2$ is a distribution $\tht \in L^\infty ((0,T);\Hdot^{\frac{\be}{2}-1}(\RR^2))\cap L^2 ((0,T);\Hdot^{\al+\frac{\be}{2}-1}(\RR^2))$ which satisfies in $\mathcal{D}'((0,T)\times \RR^2)$:
    \[\bdy_t \tht+\nabla\cdot(u\tht)+\nu(-\Delta)^\al \tht=0.\]
\end{definition}
\begin{remark}
For the rest of the paper, we adopt the convention that $C$ denotes a positive constant which may change from line-to-line. Dependencies on other parameters may be specified when they are relevant or necessary.
\end{remark}
 \end{proof}
 \section{Structure of the nonlinearity}
 In this section, we discuss the crucial observation that permits the representation of the nonlinear term in terms of commutators and makes precise its definition as a well-defined tempered distribution.
 \begin{lemma}
     Let $\tht \in L^2(\RR^2)$ such that $\hat{\tht}$ is supported away from $0$ and $\varphi \in \mathcal{D}(\RR^2)$. Then, we have 
     \begin{align}\label{identity:commutator}
          \int_{\RR^2} \tht (\Lam^{\be-2}\nabla^{\perp}\tht) \cdot \nabla \varphi\,dx=-\int_{\RR^2} \tht [\Lam^{\frac{\be}{2}-1}\nabla^{\perp}\cdot,\nabla \varphi](\Lam^{\frac{\be}{2}-1}\tht) \,dx.
     \end{align}
 \end{lemma}
 \begin{proof}
     We have
     \begin{align*}
         \int_{\RR^2} \tht (\Lam^{\be-2}\nabla^{\perp}\tht) \cdot \nabla \varphi\,dx&=\int_{\RR^2} \tht (\Lam^{\be-2}\nabla^{\perp}\tht) \cdot \nabla \varphi\,dx+\int_{\RR^2} \varphi (\Lam^{\frac{\be}{2}-1}\nabla^{\perp}\tht) \cdot (\Lam^{\frac{\be}{2}-1}\nabla \tht)\,dx\\
     \end{align*}
 Let $\rho \in \mathcal{D}$ be such that $\int_{\RR^2} \rho \,dx=1$. Define by $\rho_\eps(x):=\eps^{-2}\rho(\eps^{-1}x)$. Then , we have
 \begin{align*}
     \int_{\RR^2} \varphi (\Lam^{\frac{\be}{2}-1}\nabla^{\perp}(\rho_\eps *\tht)) \cdot (\Lam^{\frac{\be}{2}-1}\nabla (\rho_\eps *\tht))\,dx&=-\int_{\RR^2} (\Lam^{\frac{\be}{2}-1}(\rho_\eps *\tht))( \Lam^{\frac{\be}{2}-1}\nabla^{\perp}(\rho_\eps*\tht)) \cdot \nabla \varphi\,dx\\
     &=-\int_{\RR^2} (\rho_\eps * \tht)\Lam^{\frac{\be}{2}-1}\nabla^{\perp}\cdot((\Lam^{\frac{\be}{2}-1}(\rho_\eps*\tht))\nabla \varphi)\,dx.
 \end{align*}
 Letting $\eps$ go to $0$, we get
 \begin{align*}
     \int_{\RR^2} \varphi (\Lam^{\frac{\be}{2}-1}\nabla^{\perp}\tht) \cdot (\Lam^{\frac{\be}{2}-1}\nabla \tht)\,dx=-\int_{\RR^2} \tht \Lam^{\frac{\be}{2}-1}\nabla^{\perp}\cdot ((\Lam^{\frac{\be}{2}-1}\tht)\nabla \varphi)
 \end{align*}
  \end{proof}
  The identity in \eqref{identity:commutator} allows us to define the nonlinear term $\nabla \cdot(u\tht)$ as a well-defined tempered distribution. This is made precise in  \cref{P:nonlinear:distribution}.
  \begin{proposition}\label{P:nonlinear:distribution} For $f\in \Hdot^{\frac{\be}{2}-1}(\RR^2)$, define the tempered distribution $\nabla \cdot (f (\Lam^{\be-2}\nabla^{\perp}f))$ by
  \begin{align*}
     \nabla \cdot (f (\Lam^{\be-2}\nabla^{\perp}f))=\nabla \cdot (f (\Lam^{\be-2}\nabla^{\perp}S_j f)) +\nabla \cdot (S_j f (\Lam^{\be-2}\nabla^{\perp}H_j f))+\nabla \cdot (H_j f (\Lam^{\be-2}\nabla^{\perp}H_j f)),
  \end{align*}
  where $j\in \mathbb{Z}$ and $\nabla \cdot (H_j f (\Lam^{\be-2}\nabla^{\perp}H_j f))$ is the distribution defined by
  \begin{align*}
      \lb \nabla \cdot (H_j f (\Lam^{\be-2}\nabla^{\perp}H_j f))|\varphi\rb=-\int_{\RR^2} H_j f [\Lam^{\frac{\be}{2}-1}\nabla^{\perp}\cdot,\nabla \varphi](\Lam^{\frac{\be}{2}-1}H_j f) \,dx.
  \end{align*}
  for all $\varphi \in \mathcal{D}(\RR^2)$.
  \end{proposition}
  \begin{proof}
      It is straightforward to check that the first two terms are well defined tempered distributions. For the last term involving high frequencies, we apply \cref{L:mainlemma} with $\si=1-\be/2$ and $\nu=0$ to obtain
      \begin{align}\label{highfreq:bound}
          |\lb \nabla \cdot (H_j f (\Lam^{\be-2}\nabla^{\perp}H_j f))|\varphi\rb|&\le C(\Sob{\mathcal{F}(\Lam \nabla \varphi)}{L^1}+\Sob{\Lam^2 \nabla \varphi}{L^2(\RR^2)})\Sob{H_j f}{\Hdot^{\frac{\be}{2}-1}(\RR^2)}^2\notag\\
          &\le C\Sob{\varphi}{H^{3+\delta}(\RR^2)}\Sob{H_j f}{\Hdot^{\frac{\be}{2}-1}(\RR^2)}^2,\quad \delta>0.
      \end{align} 
  \end{proof}
Clearly, from the above proof, if $f$ is a function satisfying the hypotheses of the proposition, then, we have
\begin{align*}
    \nabla \cdot (f (\Lam^{\be-2}\nabla^{\perp}f))=\lim_{j \to +\infty}\left(\nabla \cdot (f (\Lam^{\be-2}\nabla^{\perp}S_j f)) +\nabla \cdot (S_j f (\Lam^{\be-2}\nabla^{\perp}H_j f))\right).
\end{align*}
\section{ Proof of \cref{T:main}}
In this section, we present the proof of the main result \cref{T:main}. Consider the following artificial viscosity approximation of \eqref{gSQG}:
\begin{align}\label{gSQG:approximation}
    \begin{cases}
        &\bdy_t \tht+\nu(-\Delta)^{\al}\tht-\eps\Delta \tht+u\cdot \nabla \tht=0,\\
        &u=\left(\bdy_{x_2}\Lam^{\beta-2}\tht,-\bdy_{x_1}\Lam^{\beta-2}\tht\right),\\
        &\tht(x,0)=(\rho_\eps *\tht_0)(x),\quad x\in \mathbb{R}^2,\,t>0,
    \end{cases}
\end{align}
where $\rho_\eps(x)=\eps^{-2}\rho(\eps^{-1}x)$, $\rho \in \mathcal{D}(\RR^2)$ and $\int_{\RR^2} \rho \,dx=1$. 
\begin{proposition}\label{P:approx:solution}
    Let $\be\in (0,2)$ and $\eps>0$. Then, the Cauchy problem for the approximate gSQG equation \eqref{gSQG:approximation} admits a unique solution $\tht^{\eps}(x,t)$ such that $\tht^{\eps}\in C([0,\infty);H^\si(\RR^2))\cap C^{\infty}((0,\infty)\times \RR^2)$ for any $\si>2$.
\end{proposition}
Note that for any $\si\ge0$ 
\[\Sob{\rho_\eps *\tht_0}{H^\si(\RR^2)}\le C(\si,\eps)\Sob{\tht_0}{\Hdot^{\frac{\be}{2}-1}(\RR^2)},\]
The proof of \cref{P:approx:solution} follows by following a similar approach to Resnick's proof for the $L^2$-solutions of the SQG. The presence of a non-zero order singularity in the constitutive laws makes the analysis more delicate and requires usage of commutator estimates. We refer the reader to \cite{ChaeConstantinCordobaGancedoWu2012, MiaoXue2011JDE, LazarXue2019} for the proof of \cref{P:approx:solution}. Note that the family of smooth solutions, denoted $\tht^\eps$, clearly satisfies the uniform bound
\begin{align*}
    \Sob{\tht^{\eps}(t)}{\Hdot^{\frac{\be}{2}-1}(\RR^2)}^2+2\nu\int_0^t \int_{\RR^2} |\Lam^{\al+\frac{\be}{2}-1}\tht^{\eps}(s)|^2\,dx\,ds\le \Sob{\tht_0}{\Hdot^{\frac{\be}{2}-1}(\RR^2)}^2.
\end{align*}
Since $L^\infty ([0,\infty);\Hdot^{\frac{\be}{2}-1}(\RR^2))$ is the dual of the separable Banach space $L^1 ([0,\infty);\Hdot^{1-\frac{\be}{2}}(\RR^2))$, we can extract from $\{\tht^\eps\}_{\eps>0}$ a subsequence $\{\tht^{\eps_k}\}_{k\ge 0}$ which is $*$-weakly convergent to some function $\tht$ in $L^\infty ([0,\infty);\Hdot^{\frac{\be}{2}-1}(\RR^2))$ (as $k\to \infty$ and $\eps_k \to 0$) and also in $\mathcal{D}'([0,\infty)\times \RR^2)$. Note that the weak convergence is not enough to conclude the convergence of the nonlinear term since we only obtain $\bdy_t \tht+\lim_{k\to \infty}(\nabla \cdot (u^{\eps_k}\tht^{\eps_k}))=0$ in $\mathcal{D}'([0,\infty)\times \RR^2)$.\par
Let  $\varphi \in \mathcal{D}([0,\infty)\times \RR^2)$ be a test function and let $R, T \in [0,\infty)$ be such that $\supp \varphi \subset (0,T)\times \mathcal{B}(R)$. We will show that as $k \to \infty$, we have
\begin{align}\label{nonlinear:convergence}
    \int_0^T \int_{\RR^2} (\tht^{\eps_k}u^{\eps_k})\cdot \nabla \varphi \,dx\,dt\rightarrow \int_0^T \int_{\RR^2} (\tht u)\cdot \nabla \varphi \,dx\,dt.
\end{align}
Let $\psi^\eps$ be the stream function associated to $\tht^\eps$. Then, we have
\begin{align}\label{def:streamfunction}
    u^\eps=\nabla^\perp \psi^\eps,\quad and \quad \tht^\eps=\Lam^{2-\beta}\psi^\eps.
\end{align}
Let $j\ge 1$ be a fixed integer. Broadly, our approach involves splitting the nonlinear terms in the frequency space at the support of the $j^{\text{th}}$ dyadic block and obtaining estimates for the resulting terms. We then take the limit in $k$ first and then $j$ later to obtain the required convergence. Since $\tht^\eps$ is uniformly bounded in $\Hdot^{\frac{\be}{2}-1}(\RR^2)$, we find that $H_j \psi^\eps$ is uniformly bounded in $H^{1-\frac{\be}{2}}(\RR^2)$ with 
\begin{align}\label{psi:bound}
\Sob{H_j \psi^\eps}{H^{1-\frac{\be}{2}}(\RR^2)}\le C(j,\be)\Sob{\tht^\eps}{\Hdot^{\frac{\be}{2}-1}(\RR^2)}\le C(j,\be)\Sob{\tht_0}{\Hdot^{\frac{\be}{2}-1}(\RR^2)}
\end{align}
for all $\eps>0$ where $C(j,\be)$ is a positive constant. We will now obtain an estimate for $\bdy_t \tht^\eps$. 
Let $\phi \in \mathcal{D}(\RR^2)$. Upon integrating by parts and applying the Cauchy-Schwarz inequality, we have
\begin{align}\label{dt:theta:bound}
    \left|\int_{\RR^2}\bdy_t \tht^\eps \phi \,dx\right|&\le \left|\int_{\RR^2}( u^\eps\cdot \nabla \tht^\eps) \phi\,dx\right|+\eps\left|\int_{\RR^2}\Delta \tht^\eps \phi\,dx\right|+\nu\left|\int_{\RR^2}(-\Delta)^{\al}\tht^\eps \phi\,dx\right|\notag\\&\le \left|\int_{\RR^2}(\tht^\eps u^\eps)\cdot \nabla \phi\,dx\right|+\eps\left|\int_{\RR^2}\Lam^{\frac{\be}{2}-1}\tht^\eps \Lam^{3-\frac{\be}{2}} \phi\,dx\right|+\nu\left|\int_{\RR^2}\Lam^{\frac{\be}{2}-1}\tht^\eps \Lam^{2\al+1-\frac{\be}{2}}\phi\,dx\right|\notag\\
    &\le \left|\int_{\RR^2}(\tht^\eps u^\eps)\cdot \nabla \phi\,dx\right|+ \Sob{\tht_0}{\Hdot^{\frac{\be}{2}-1}(\RR^2)}\Sob{\phi}{H^{3-\frac{\be}{2}}(\RR^2)}
\end{align}
Using \eqref{def:streamfunction}, decomposing the trilinear term, and integrating by parts, we have
\begin{align*}
    &\int_{\RR^2}(\tht^\eps u^\eps)\cdot \nabla \phi\,dx\\
    &=\int_{\RR^2} \tht^\eps \nabla^\perp \Lam^{\be-2}S_j \tht^\eps \cdot \nabla \phi\,dx+\int_{\RR^2}S_j \tht^\eps \nabla^\perp \Lam^{\be-2}H_j \tht^\eps\cdot \nabla \phi\,dx+\int_{\RR^2}H_j \tht^\eps \nabla^\perp \Lam^{\be-2}H_j \tht^\eps\cdot \nabla \phi\,dx\\
    &=\int_{\RR^2} \tht^\eps \nabla^\perp \Lam^{\be-2}S_j \tht^\eps \cdot \nabla \phi\,dx-\int_{\RR^2}H_j \psi^\eps \nabla^\perp S_j \tht^\eps\cdot \nabla \phi\,dx+\int_{\RR^2}H_j \tht^\eps \nabla^\perp \Lam^{\be-2}H_j \tht^\eps\cdot \nabla \phi\,dx\\
    &=J_1(\tht^\eps)(t)+J_2(\tht^\eps)(t)+J_3(\tht^\eps)(t), 
\end{align*}
We now estimate $J_1(\tht^\eps)$, $J_2(\tht^\eps)$, and $J_3(\tht^\eps)$ individually. Applying \cref{L:Sobolev:product} with $s=t=1-\be/4$, we have
\begin{align}\label{dt:theta:bound:J1}
    |J_1(\tht^\eps)|&\le \Sob{\tht^\eps}{\Hdot^{\frac{\be}{2}-1}(\RR^2)}\Sob{\mathcal{R}^\perp \Lam^{\be-1}S_j \tht^\eps \cdot \nabla \phi}{\Hdot^{1-\frac{\be}{2}}(\RR^2)}\notag\\
    &\le \Sob{\tht^\eps}{\Hdot^{\frac{\be}{2}-1}(\RR^2)}\Sob{\mathcal{R}^\perp S_j \tht^\eps}{\Hdot^{\frac{3\be}{4}}(\RR^2)}\Sob{\nabla \phi}{\Hdot^{1-\frac{\be}{4}}(\RR^2)}\notag\\
    &\le C(j,\be)\Sob{\tht_0}{\Hdot^{\frac{\be}{2}-1}(\RR^2)}^2\Sob{\phi}{H^2(\RR^2)}.
\end{align}
We estimate $J_2(\tht^\eps)$ by using \eqref{psi:bound} as follows
\begin{align}\label{dt:theta:bound:J2}
    |J_2(\tht^\eps)|&\le \Sob{H_j\psi^\eps}{L^2(\RR^2)}\Sob{\nabla^\perp S_j \tht^\eps}{L^2(\RR^2)}\Sob{\nabla \phi}{L^\infty(\RR^2)}\notag\\
    &\le C(j,\be)\Sob{\tht_0}{\Hdot^{\frac{\be}{2}-1}(\RR^2)}^2\Sob{\phi}{H^3(\RR^2)}.
\end{align}
For $J_3(\tht^\eps)$, we use \eqref{highfreq:bound} with $f=\tht^\eps$ to obtain
\begin{align}\label{dt:theta:bound:J3}
    |J_3(\tht^\eps)|&\le C\Sob{H_j \tht^\eps}{\Hdot^{\frac{\be}{2}-1}(\RR^2)}^2 \Sob{\phi}{H^{3+\delta}(\RR^2)}\notag\\
    &\le C\Sob{\tht_0}{\Hdot^{\frac{\be}{2}-1}(\RR^2)}^2\Sob{\phi}{H^4(\RR^2)}.
\end{align}
From \eqref{dt:theta:bound}-\eqref{dt:theta:bound:J3}, we conclude
\begin{align*}
    \bdy_t \tht^\eps \in L^\infty([0,T];H^{-4}(\RR^2)),\\
\end{align*}
uniformly in $\eps$. Using this, we further deduce that
\begin{align}\label{dt:psi:bound}
    \bdy_t H_j\psi^\eps \in L^\infty([0,T];H^{-4}(\RR^2))
\end{align}
uniformly in $\eps$. Indeed, we have
\begin{align*}
    \left|\int_{\RR^2}\bdy_t H_j \psi^\eps \phi\,dx\right|&=\left|\int_{\RR^2}\bdy_t \tht^\eps \Lam^{\be-2}H_j \phi\,dx\right|\\
    &\le C\Sob{\bdy_t \tht^\eps}{H^{-4}(\RR^2)}\Sob{\Lam^{\be-2}H_j \phi}{H^4(\RR^2)}\\
    &\le C(j,\be)\Sob{\phi}{H^4(\RR^2)}.
\end{align*}
From \eqref{psi:bound} and \eqref{dt:psi:bound}, weak convergence of $\tht^{\eps_k}$, and by an application of the Aubin-Lions-Simon lemma, we extract a subsequence (still denoted by $H_j \psi^\eps_k$) such that
\begin{align}\label{H_jpsi:L2:convergence}
    H_j \psi^{\eps_k}\rightarrow H_j \psi \quad \text{in} \quad L^\infty([0,T];L^2_{\text{loc}}(\RR^2)),
\end{align}
where $H_j\psi\in L^\infty ((0,T);H^{1-\frac{\be}{2}}(\RR^2))$
We will now establish \eqref{nonlinear:convergence}. Let 
$\rho\in \mathcal{D}(\RR^2)$ be such that $\rho\equiv 1$ on $\mathcal{B}(R)$ with $\supp \rho \subset \mathcal{B}(2R)$. We define $\gam=\rho(x/10)$.
We decompose $\int_0^T \int_{\RR^2} (\tht^{\eps_k}u^{\eps_k})\cdot \nabla \varphi \,dx\,dt$ into four terms as follows
\begin{align*}
   \int_0^T \int_{\RR^2}  (\tht^{\eps_k}u^{\eps_k})\cdot \nabla \varphi \,dx\,dt=& \int_0^T \int_{\RR^2}  \tht^{\eps_k}\rho(\Lam^{\be-2}\nabla^{\perp}((1-\gam)S_j\tht^{\eps_k}))\cdot \nabla \varphi \,dx\,dt\\&+\int_0^T \int_{\RR^2}  \tht^{\eps_k}\gam(\Lam^{\be-2}\nabla^{\perp}(\gam S_j\tht^{\eps_k}))\cdot \nabla \varphi \,dx\,dt \\&+\int_0^T \int_{\RR^2}  S_j \tht^{\eps_k}(\Lam^{\be-2}\nabla^{\perp}H_j\tht^{\eps_k})\cdot \nabla \varphi \,dx\,dt\\
    &+\int_0^T \int_{\RR^2}  H_j \tht^{\eps_k}(\Lam^{\be-2}\nabla^{\perp}H_j\tht^{\eps_k})\cdot \nabla \varphi \,dx\,dt\\=&J_1^a (\tht^{\eps_k})+J_1^b(\tht^{\eps_k})+J_2(\tht^{\eps_k})+J_3(\tht^{\eps_k}).
\end{align*}
We first prove that $J_1^a(\tht^{\eps_{k}})\rightarrow J_1^a(\tht)$. 
We have
\begin{align*}
    J_1^a (\tht^{\eps_{k}})- J_1^a (\tht)=&\int \int \tht^{\eps_k}\rho(\Lam^{\be-2}\nabla^{\perp}((1-\gam)(S_j\tht^{\eps_k}-S_j\tht)\cdot \nabla \varphi \,dx\,dt\\
    &+\int \int (\tht^{\eps_k}-\tht)\rho(\Lam^{\be-2}\nabla^{\perp}((1-\gam)(S_j\tht^{\eps_k}))\cdot \nabla \varphi \,dx\,dt\\
    =&\til{J}_{11}^a(\tht^{\eps_k},\tht)+\til{J}_{12}^a(\tht^{\eps_k},\tht).
\end{align*}
Let us denote by
\[\zeta^{\eps_k}_{j}\equiv\rho(\Lam^{\be-2}\nabla^{\perp}((1-\gam)S_j\tht^{\eps_k})),\]
then we claim that
\begin{align*}
    \zeta^{\eps_k}_{j} \rightarrow \zeta_{j}\equiv \rho(\Lam^{\be-2}\nabla^{\perp}((1-\gam)S_j\tht))\quad \text{in}\quad L^\infty([0,T];H^{1-\frac{\be}{4}}(\mathcal{B}(R))).
\end{align*}
First, we observe that for all $x\in \mathcal{B}(2R)$,
\begin{align*}
    |\zeta^{\eps_k}_{j}(x)|&\le C\int_{|y|\ge 10R}\frac{1}{|x-y|^{1+\be}}|S_j \tht^{\eps_k}(y)|\,dy\\
    &\le C\int_{|y|\ge 10R}\frac{1}{|y|^{1+\be}}||S_j \tht^{\eps_k}(y)|\,dy\\
    &\le C\frac{1}{R^\be}\Sob{S_j \tht}{L^2(\RR^2)}\le C(j,\be)\Sob{\tht_0}{\Hdot^{\frac{\be}{2}-1}(\RR^2)}.
\end{align*}
Similarly, we obtain for all $x\in \mathcal{B}(2R)$.  
\begin{align*}
    |\nabla \zeta^{\eps_k}_{j}(x)|\le C(j,\be)\Sob{\tht_0}{\Hdot^{\frac{\be}{2}-1}(\RR^2)}.  
\end{align*}
Thus, we have
\begin{align*}
    \Sob{ \zeta^{\eps_k}_{j}}{L^\infty([0,T];H^1(\RR^2))}\le C(j,\be)\Sob{\tht_0}{\Hdot^{\frac{\be}{2}-1}(\RR^2)}.
\end{align*}
We now obtain a bound on $\bdy_t \zeta^{\eps_k}_{j}$. Let $\varphi\in \mathcal{D}(\RR^2)$. We have
\begin{align*}
    \left | \int \bdy_t \zeta^{\eps_k}_{j} \phi\,dx\right|&=\left|\int \rho(\Lam^{\be-2}\nabla^{\perp}((1-\gam)S_j\bdy_t\tht^{\eps_k})) \phi \,dx\right|\\
    &=\left| \int \bdy_t\tht^{\eps_k} S_j((1-\gam)\nabla^\perp \Lam^{\be-2}(\rho \phi))\,dx\right|\\
    &\le \Sob{\bdy_t \tht^{\eps_k}}{L^\infty([0,T];H^{-4}(\RR^2))}\Sob{S_j((1-\gam)\nabla^\perp \Lam^{\be-2}(\rho \phi))}{H^4(\RR^2)}\\
   & \le C(j,\be)\Sob{\phi}{L^2(\RR^2)},
\end{align*}
where in the last line, we used
\begin{align*}
    \Sob{S_j((1-\gam)\nabla^\perp \Lam^{\be-2}(\rho \phi))}{H^4(\RR^2)}&\le C(j)\Sob{\int_{\mathcal{B}(2R)}\frac{1}{|x-y|^{1+\be}}|\rho(y) \phi(y)|\,dy}{L^2(\mathcal{B}(10R)^c)}\\
    &\le C(j)\Sob{|x|^{-1-\be}}{L^2(\mathcal{B}(10R)^c)}\int_{\mathcal{B}(2R)}|\rho \phi|\,dy\\
    &\le C(j,\be)\Sob{\phi}{L^2(\RR^2)}.
\end{align*}
By an application of the Aubin-Lions-Simon lemma, we extract a further subsequence (still denoted by $\zeta^{\eps_k}_{j}$) such that
\begin{align}\label{zeta_j:s:convergence}
\zeta^{\eps_k}_{j}\rightarrow \zeta_{j}\quad \text{in}\quad L^\infty([0,T];H^{1-\frac{\be}{4}}(\mathcal{B}(R)).
\end{align}
Applying \eqref{L:Sobolev:product} with $s=t=1-\frac{\be}{4}$ and using \eqref{zeta_j:s:convergence}, we obtain
\begin{align}\label{lim:J11a}
    &\lim_{k\to \infty}|\til{J}_{11}^a(\tht^{\eps_k},\tht)|\notag\\&\le \lim _{k\to \infty}C\Sob{\tht^{\eps_k}}{L^{2}([0,T];\Hdot^{\frac{\be}{2}-1}(\RR^2))}\Sob{\zeta_{j}^{\eps_k}-\zeta_{j}}{L^2([0,T];H^{1-\frac{\be}{4}}(\mathcal{B}(R)))}\Sob{\nabla \varphi}{L^\infty([0,T];H^{1-\frac{\be}{4}}(\RR^2))}\notag\\
    &\le CT\Sob{\tht_0}{\Hdot^{\frac{\be}{2}-1}(\RR^2)}\lim_{k \to \infty}\Sob{\zeta_{j}^{\eps_k}-\zeta_{j}}{L^\infty([0,T];H^{1-\frac{\be}{4}}(\mathcal{B}(R)))}=0.
\end{align}
The weak convergence of $\tht^{\eps_k}$ in $L^\infty([0,T];\Hdot^{\frac{\be}{2}-1}(\RR^2))$ along with \eqref{zeta_j:s:convergence} directly gives us 
\begin{align}\label{lim:J12a}
    \lim_{k\to \infty}|\til{J}_{12}^a(\tht^{\eps_k},\tht)|=0.
\end{align}
Now we prove that $J_1^b(\tht^{\eps_{k}})\rightarrow J_1^b(\tht)$. 
We have
\begin{align*}
    J_1^b (\tht^{\eps_{k}})- J_1^b (\tht)=&\int \int \tht^{\eps_k}\gam(\Lam^{\be-2}\nabla^{\perp}(\gam(S_j\tht^{\eps_k})-S_j\tht)\cdot \nabla \varphi \,dx\,dt\\
    &+\int \int (\tht^{\eps_k}-\tht)\gam(\Lam^{\be-2}\nabla^{\perp}(\gam(S_j\tht^{\eps_k}))\cdot \nabla \varphi \,dx\,dt\\
    =&\til{J}_{11}^b(\tht^{\eps_k},\tht)+\til{J}_{12}^b(\tht^{\eps_k},\tht).
\end{align*}
Let us denote by
\[\bar{\zeta}^{\eps_k}_{j}\equiv\gam(\Lam^{\be-2}\nabla^{\perp}(\gam S_j\tht^{\eps_k})),\]
then we claim that
\begin{align*}
    \bar{\zeta}^{\eps_k}_{j} \rightarrow \bar{\zeta}_j\equiv \gam(\Lam^{\be-2}\nabla^{\perp}(\gam S_j\tht))\quad \text{in}\quad L^\infty([0,T];H^{1-\frac{\be}{4}}(\mathcal{B}(R))).
\end{align*}
First, we observe that
\begin{align*}
    \Sob{\bar{\zeta}^{\eps_k}_{j}}{L^\infty ([0,T];H^{1}(\RR^2))}&\le C\Sob{\Lam^{\be-1}S_0(\gam S_j \tht^{\eps_k})}{L^\infty ([0,T];L^{2}(\RR^2))}+C\Sob{\Lam^{\be}H_0(\gam S_j \tht^{\eps_k})}{L^\infty ([0,T];L^{2}(\RR^2))}\\
    &\le \begin{cases}
        C\Sob{\gam S_j\tht^{\eps_k}}{L^\infty([0,T];L^2(\RR^2))}+C\Sob{\gam S_j \tht^{\eps_k}}{L^\infty ([0,T];H^{2}(\RR^2))},\, &\be\in [1,2)\\
        C\Sob{\gam S_j\tht^{\eps_k}}{L^\infty([0,T];L^{\frac{2}{2-\be}}(\RR^2))}+C\Sob{\gam S_j \tht^{\eps_k}}{L^\infty ([0,T];H^{2}(\RR^2))},\, &\be\in (0,1)
    \end{cases}\\
    &\le C\Sob{\gam}{H^2(\RR^2)}\Sob{ S_j\tht^{\eps_k}}{L^\infty([0,T];H^2(\RR^2))}\\
    &\le C(j,\be)\Sob{\tht^{\eps_k}}{L^{\infty}([0,T];\Hdot^{\frac{\be}{2}-1}(\RR^2))} \le C(j,\be)\Sob{\tht_0}{\Hdot^{\frac{\be}{2}-1}(\RR^2)}.
\end{align*}
We now obtain a bound on $\bdy_t \bar{\zeta}^{\eps_k}_{j}$. Let $\phi\in \mathcal{D}(\RR^2)$. We have
\begin{align*}
    \left | \int \bdy_t \bar{\zeta}^{\eps_k}_{j} \phi\,dx\right|&=\left|\int \gam(\Lam^{\be-2}\nabla^{\perp}(\gam S_j\bdy_t\tht^{\eps_k})) \phi \,dx\right|\\
    &=\left| \int \bdy_t\tht^{\eps_k} S_j(\gam \nabla^\perp \Lam^{\be-2}(\gam \phi))\,dx\right|\\
    &\le \Sob{\bdy_t \tht^{\eps_k}}{L^\infty([0,T];H^{-4}(\RR^2))}\Sob{S_j(\gam \nabla^\perp \Lam^{\be-2}(\gam \phi))}{H^4(\RR^2)}\\
   & \le C(j)\Sob{\gam \nabla^\perp \Lam^{\be-2}(\gam \phi) }{L^2(\RR^2)}\\
   &\le \begin{cases}
   C(j)\Sob{\Lam^{\be-1}(\gam \phi)}{L^{2}(\RR^2)},\,\quad &\be \in [1,2)\\
       C(j)\Sob{\Lam^{\be-1}(\gam \phi)}{L^{\frac{2}{\be}}(\RR^2)},\,\quad &\be \in (0,1)
   \end{cases}  \\
   &\le C(j,\be)\Sob{\phi }{H^1(\RR^2)}.
\end{align*}
By an application of the Aubin-Lions-Simon lemma, we can extract a further subsequence such that
\begin{align}\label{zeta_j:bar:s:convergence}
\bar{\zeta}^{\eps_k}_{j}\rightarrow \bar{\zeta}_{j}\quad \text{in}\quad L^\infty([0,T];H^{1-\frac{\be}{4}}(\mathcal{B}(R)).
\end{align}
Applying \eqref{L:Sobolev:product} with $s=t=1-\frac{\be}{4}$ and using \eqref{zeta_j:bar:s:convergence}, we obtain
\begin{align}\label{lim:J11b}
    &\lim_{k\to \infty}|\til{J}_{11}^b(\tht^{\eps_k},\tht)|\notag\\&\le \lim _{k\to \infty}C\Sob{\tht^{\eps_k}}{L^{2}([0,T];\Hdot^{\frac{\be}{2}-1}(\RR^2))}\Sob{\bar{\zeta}_j^{\eps_k}-\bar{\zeta}_j}{L^2([0,T];H^{1-\frac{\be}{4}}(\mathcal{B}(R)))}\Sob{\nabla \varphi}{L^\infty([0,T];H^{1-\frac{\be}{4}}(\RR^2))}\notag\\
    &\le CT\Sob{\tht_0}{\Hdot^{\frac{\be}{2}-1}(\RR^2)}\lim_{k \to \infty}\Sob{\bar{\zeta}_j^{\eps_k}-\bar{\zeta}_j}{L^\infty([0,T];H^{1-\frac{\be}{4}}(\mathcal{B}(R)))}=0.
\end{align}
The weak convergence of $\tht^{\eps_k}$ in $L^\infty([0,T];\Hdot^{\frac{\be}{2}-1}(\RR^2))$ along with \eqref{zeta_j:bar:s:convergence} gives us 
\begin{align}\label{lim:J12b}
    \lim_{k\to \infty}|\til{J}_{12}^b(\tht^{\eps_k},\tht)|=0.
\end{align}
Now we prove that $J_2(\tht^{\eps_{k}})\rightarrow J_2(\tht)$. 
We have
\begin{align*}
    J_2 (\tht^{\eps_{k}})- J_2 (\tht)=&\int \int H_j\psi^{\eps_k}(\nabla^{\perp}S_j\tht^{\eps_k}-\nabla^{\perp}S_j\tht)\cdot \nabla \varphi \,dx\,dt\\
    &+\int \int (H_j\psi^{\eps_k}-H_j \psi)\nabla^{\perp}S_j\tht\cdot \nabla \varphi \,dx\,dt\\
    =&\til{J}_{21}(\tht^{\eps_k},\tht)+\til{J}_{22}(\tht^{\eps_k},\tht).
\end{align*}
Let us denote by
\[\varrho^{\eps_k}_{j}\equiv\nabla^{\perp}S_j\tht^{\eps_k},\]
then, we claim that
\begin{align*}
    \varrho^{\eps_k}_{j} \rightarrow \varrho_j\equiv \nabla^{\perp}S_j\tht\quad \text{in}\quad L^\infty([0,T];L^2_{\text{loc}}(\RR^2)).
\end{align*}
We have
\begin{align*}
    \Sob{\varrho^{\eps_k}_{j}}{L^\infty ([0,T];H^1(\RR^2))}\le C(j,\be)\Sob{S_j \tht^{\eps_k}}{L^\infty([0,T];\Hdot^{\frac{\be}{2}-1}(\RR^2))}\le C(j,\be)\Sob{\tht_0}{\Hdot^{\frac{\be}{2}-1}(\RR^2)}.
\end{align*}
We now obtain a bound on $\bdy_t \varrho^{\eps_k}_{j}$. Let $\phi\in \mathcal{D}(\RR^2)$. We have
    \begin{align*}
    \left | \int \bdy_t \varrho^{\eps_k}_{j} \phi\,dx\right|&=\left|\int (\nabla^{\perp}S_j \bdy_t\tht^{\eps_k}) \phi \,dx\right|\\
    &=\left| \int \bdy_t\tht^{\eps_k} S_j(\nabla^\perp \phi)\,dx\right|\\
    &\le \Sob{\bdy_t \tht^{\eps_k}}{L^\infty([0,T];H^{-4}(\RR^2))}\Sob{S_j\nabla^\perp \phi}{H^4(\RR^2)}\\
   & \le C(j)\Sob{\phi}{H^1(\RR^2)},
\end{align*}
By an application of the Aubin-Lions-Simon lemma, we can extract a further subsequence such that
\begin{align}\label{rho_j:s:convergence}
\varrho^{\eps_k}_{j}\rightarrow \varrho_{j}\quad \text{in}\quad L^\infty([0,T];L^2_{\text{loc}}(\RR^2)).
\end{align}
Applying Holder's inequality, and using \eqref{rho_j:s:convergence}, we obtain
\begin{align}\label{lim:J21}
    \lim_{k\to \infty}|\til{J}_{21}(\tht^{\eps_k},\tht)|&\le \lim _{k\to \infty}C\Sob{H_j\psi^{\eps_k}}{L^{2}([0,T];L^2(\RR^2))}\Sob{\varrho_j^{\eps_k}-\varrho_j}{L^2([0,T];L^2(\mathcal{B}(R)))}\Sob{\nabla \varphi}{L^\infty([0,T];L^\infty(\RR^2))}\notag\\
    &\le CT\Sob{H_j \psi_0}{L^2(\RR^2)}\lim_{k \to \infty}\Sob{\varrho_j^{\eps_k}-\varrho_j}{L^\infty([0,T];L^2(\mathcal{B}(R)))}=0.
\end{align}
By \eqref{H_jpsi:L2:convergence} and \eqref{rho_j:s:convergence}, we obtain
\begin{align}\label{lim:J22}
    \lim_{k\to \infty}|\til{J}_{22}(\tht^{\eps_k},\tht)|=0.
\end{align}
We now obtain an estimates for $J_3(\tht^{\eps_k})$ and $J_3(\tht)$. By applying \eqref{highfreq:bound} with $f=\tht^{\eps_k}$, we obtain
\begin{align*}
    \left|\int_{\RR^2}  H_j \tht^{\eps_k}(\Lam^{\be-2}\nabla^{\perp}H_j\tht^{\eps_k})\cdot \nabla \varphi \,dx\,dt\right|\le C \Sob{\varphi}{H^{3+\delta}(\RR^2)}\Sob{H_j \tht^{\eps_k}}{\Hdot^{\frac{\be}{2}-1}(\RR^2)}^2.
\end{align*}We have
\begin{align}\label{est:J3}
    |J_3(\tht^{\eps_k})|&\le C\int_0^\infty \Sob{\varphi}{H^{3+\delta}(\RR^2)}\Sob{H_j \tht^{\eps_k}}{\Hdot^{\frac{\be}{2}-1}(\RR^2)}^2\,dt\notag\\
    &\le C(\varphi)2^{-2\al j}\int_0^\infty \Sob{H_j \tht^{\eps_k}}{\Hdot^{\al+\frac{\be}{2}-1}(\RR^2)}^2\,dt
    \notag\\&\le C(\varphi)2^{-2\al j}\Sob{\tht^{\eps_k}}{L^2([0,\infty);\Hdot^{\al+\frac{\be}{2}-1}(\RR^2))}\notag\\
    &\le C(\varphi)2^{-2\al j}\Sob{\tht_0}{\Hdot^{\frac{\be}{2}-1}(\RR^2)}^2.
\end{align}
Similarly, we have
\begin{align}\label{est:J3tht}
    |J_3(\tht)|\le C(\varphi)2^{-2\al j}\Sob{\tht_0}{\Hdot^{\frac{\be}{2}-1}(\RR^2)}^2.
\end{align}
From \eqref{lim:J11a}, \eqref{lim:J12a}, \eqref{lim:J11b}, \eqref{lim:J12b}, and \eqref{lim:J21}-\eqref{est:J3tht}, we obtain
\begin{align*}
    \limsup_{k\to \infty}\left|\int_0^T \int_{\RR^2} (\tht^{\eps_k}u^{\eps_k})\cdot \nabla \varphi \,dx\,dt- \int_0^T \int_{\RR^2} (\tht u)\cdot \nabla \varphi \,dx\,dt\right|\le C(\varphi)2^{-2\al j}\Sob{\tht_0}{\Hdot^{\frac{\be}{2}-1}(\RR^2)}^2.
\end{align*}
When $j$ goes to $+\infty$, we obtain the desired result.

\bibliographystyle{plain}
\bibliography{main_bib.bib}
\vspace{.3in}
\end{document}